\documentclass[11pt,dvipdfmx]{article}

\usepackage{amsmath,amssymb}
\usepackage{amsthm}
\usepackage[abbrev]{amsrefs}
\usepackage{latexsym}
\usepackage{txfonts}
\usepackage{graphicx}
\usepackage{tikz}
\usepackage{bm}

\title{On $p$-adic entropy of some solenoid dynamical systems}
\date{}
\author{Yu Katagiri}

\theoremstyle{definition}
\newtheorem{theorem}{Theorem}[section]
\newtheorem*{theorem*}{Theorem}
\newtheorem{definition}[theorem]{Definition}
\newtheorem*{definition*}{Definition}
\newtheorem{lemma}[theorem]{Lemma}
\newtheorem*{lemma*}{Lemma}

\newtheorem*{proposition*}{Proposition}
\newtheorem{example}[theorem]{Example}
\newtheorem*{example*}{Example}
\newtheorem{remark}[theorem]{Remark}
\newtheorem*{remark*}{Remark}
\newtheorem{corollary}[theorem]{Corollary}
\newtheorem{corollary*}{Corollary}

\begin{document}
\maketitle

\section{Introduction}

Let $d$ be a positive integer. For a $\mathbb{Z}^d$-action on a set $X$, the periodic entropy of the $\mathbb{Z}^d$-action on $X$ is defined by
\begin{align*}
h(X) \coloneqq \displaystyle \lim_{n \to \infty} \frac{1}{[\mathbb{Z}^d \colon (n \mathbb{Z})^d]} \log |{\rm Fix}_{(n \mathbb{Z})^d} (X)|
\end{align*}
if the limit exists. Here ${\rm Fix}_{(n \mathbb{Z})^d} (X)$ is the set of fixed points of the $(n\mathbb{Z})^d$-action on $X$. If $X$ is a compact metrizable abelian group and the $\mathbb{Z}^d$-action on $X$ is continuous, one can define other entropies, topological entropy and measure theoretic entropy (with respect to the normalized Haar measure on $X$). See \cite[Chapter V]{Sc95} and \cite[Appendix A]{LSW90} for the definition of these entropies.\\
\\
We consider the action of $\mathbb{Z}^d$ on a discrete abelian group $L_d(\mathbb{Z}) \coloneqq \mathbb{Z}[t_1^{\pm 1},\dots,t_d^{\pm 1}]$ given by
\begin{align}\label{actionX_f}
\delta \cdot \displaystyle \sum_{\nu \in \mathbb{Z}^d} a_\nu t^\nu = \displaystyle \sum_{\nu \in \mathbb{Z}^d} a_\nu t^{\nu + \delta}
\end{align}
for $\delta \in \mathbb{Z}^d$ and $\sum_{\nu \in \mathbb{Z}^d} a_\nu t^\nu \in L_d(\mathbb{Z})$. Here, for $t=(t_1, \cdots , t_d)$ and $\nu =(\nu_1, \cdots ,\nu_d)$, define $t^\nu=t_1^{\nu_1}\dots t_d^{\nu_d}$. Then the Pontryagin dual $\widehat{L_d(\mathbb{Z})}$ is a compact abelian group and the action on $L_d(\mathbb{Z})$ induces the action on $\widehat{L_d(\mathbb{Z})}$. For a fixed $f \in L_d(\mathbb{Z})$, the above action induces the $\mathbb{Z}^d$-actions on $L_d(\mathbb{Z})/{fL_d(\mathbb{Z})}$ and its Pontryagin dual $X_f \coloneqq (L_d(\mathbb{Z})/{fL_d(\mathbb{Z})})^\wedge$. In \cite[Theorem 3.1]{LSW90}, Lind, Schmidt and Ward showed that the (topological) entropy of $X_f$ is given by
\begin{align*}
h(X_f)=m(f)
\end{align*}
for $0 \neq f \in L_d(\mathbb{Z})$. Here $m(f)$ is the Mahler measure of $f$, which is defined by
\begin{align*}
 m(f) \coloneqq& \frac{1}{(2 \pi \sqrt{-1})^d} \int_{T^d} \log |f(z_1, \dots ,z_d)| \frac{dz_1}{z_1} \cdots \frac{dz_d}{z_d} \in \mathbb{R},
\end{align*}
and
\begin{align}\label{torus}
T^d=\{(z_1, \dots ,z_d) \in \mathbb{C}^d \mid |z_1|= \dots =|z_d|=1\}
\end{align}
is the $d$-torus. Note that the periodic entropy of $X_f$ exists and coincides with the topological entropy of $X_f$ if and only if $f$ does not vanish on $T^d$ (\cite{D12}, \cite{LSW90}).

Let $K$ be a number field with the ring $\mathcal{O}_K$ of the integers and $0 \neq f \in L_d(\mathcal{O}_K) \coloneqq \mathcal{O}_K[t_1^{\pm 1},\dots,t_d^{\pm 1}]$. In \cite{E99}, Einsiedler extended this theorem to the case $X_f \coloneqq (L_d(\mathcal{O}_K)/{fL_d(\mathcal{O}_K)})^\wedge$ and proved that
\begin{align*}
h(X_f)=m(N_{K/\mathbb{Q}}(f))
\end{align*}
holds. Here
\begin{align*}
 N_{K/\mathbb{Q}}(f)= \displaystyle \prod_{\tau : K \hookrightarrow \mathbb{C}} \tau (f) \in L_d(\mathbb{Z}).
\end{align*}
\\
\\
Let $p$ be a prime, $\mathbb{C}_p$ be the completion of $\overline{\mathbb{Q}}_p$ with $|p|_p=p^{-1}$ and $\log_p$ be the $p$-adic logarithm with $\log_p p=0$. In \cite{D09}, Deninger introduced the $p$-adic entropy as a $p$-adic analogue of (periodic) entropy. A simple definition of the $p$-adic entropy is the following (see \cite{D09} for the detailed definition).

\begin{definition}\label{def of p-adic entropy}
Assume that $\mathbb{Z}^d$ acts on a set $X$. If the limit
\begin{align*}
h_p(X) \coloneqq \displaystyle \lim_{\substack{n \to \infty \\ (n,p)=1}} \frac{1}{[\mathbb{Z}^d \colon (n \mathbb{Z})^d]} \log_p |{\rm Fix}_{(n \mathbb{Z})^d} (X)|
\end{align*}
exists, we call $h_p(X)$ the $p$-adic entropy.
\end{definition}

\begin{definition}
Let $f \in \mathbb{C}_p[t_1^{\pm 1}, \dots ,t_d^{\pm 1}]$. If the limit
\begin{align}\label{mahler}
m_p(f)= \displaystyle \lim_{\substack {N \to \infty \\ (N,p)=1}} \frac{1}{N^d} \displaystyle \sum_{\zeta \in \mu _N^d} \log_p f(\zeta)
\end{align}
exists, we call $m_p(f)$ the $p$-adic Mahler measure of $f$. Here
\begin{align*}
\mu_N&=\{z \in \mathbb{C}_p \mid z^N=1\}.
\end{align*}
Note that the R.H.S. of (\ref{mahler}) is a $p$-adic analogue of the line integral on the $p$-adic $d$-torus
\begin{align}\label{p-adic torus}
T_p^d&=\{(z_1, \dots ,z_d) \in \mathbb{C}_p^d \mid |z_1|_p= \dots =|z_d|_p=1\}.
\end{align}
See \cite{BD99} for details.
\end{definition}

Deninger proved Theorem \ref{thmD} as a $p$-adic analogue of Lind-Schmidt-Ward's theorem.

\begin{theorem}\cite[Theorem 1.1]{D09}\label{thmD}
Let $f \in L_d(\mathbb{Z})$ and assume that $f$ does not vanish at any $z \in T_p^d$. Then the $p$-adic entropy $h_p(X_f)$ of the $\mathbb{Z}^d$-action on $X_f$ exists and we have
\begin{align*}
h_p(X_f)=m_p(f).
\end{align*}
\end{theorem}

The first aim in this paper is to prove the following theorem as a $p$-adic analogue of Einsiedler's theorem. 

\begin{theorem}\label{thmK1}
Let $K$ be a number field and $\mathcal{O}_K$ be the ring of the integers of $K$. Assume that $f \in L_d(\mathcal{O}_K)$ and $N_{K/\mathbb{Q}}(f)$ does not vanish at any point of the $p$-adic torus $T_p^d$. Then the $p$-adic entropy $h_p(X_f)$ of the $\mathbb{Z}^d$-action on $X_f=(L_d(\mathcal{O}_K)/{fL_d(\mathcal{O}_K)})^\wedge$ exists and we have
\begin{align*}
h_p(X_f)=m_p(N_{K/\mathbb{Q}}(f)).
\end{align*}
Here, the action on $X_f$ is induced by the action on $L_d(\mathcal{O}_K)$ as (\ref{actionX_f}).
\end{theorem}

We will prove Theorem \ref{thmK1} in Section 2.

In Section 3, we will discuss the notion of the $p$-adic expansiveness, following Br\"auer \cite{B10}, to explain the statement of our main result. The expansiveness of the action on the dynamical system is classical. For example, it is well-known that for an expansive $\mathbb{Z}^d$-action the notions of topological entropy, measure theoretic entropy and periodic entropy coincide \cite[Theorem A.1]{LSW90}. Because we do not have $p$-adic analogues of topological entropy nor measure theoretic entropy, the $p$-adically expansiveness seems to be important.

In Section 4, we consider the $p$-adic entropy of solenoidal automorphisms. Let $S \subset P$ a subset of the set of all primes $P$. The solenoid $\Sigma_S$ is defined to be the Pontryagin dual of the discrete abelian group $\mathbb{Z}[1/{S}]$. In particular, the Pontryagin dual of $\mathbb{Z}[1/{P}]=\mathbb{Q}$ is called the full solenoid and denoted by $\Sigma$. We fix a matrix $A \in \mathrm{GL}_m (\mathbb{Z}[1/{S}])$. The $\mathbb{Z}$-action on $\mathbb{Z}[1/{S}]^m$ defined by $A$ induces the action on $\Sigma_S^m$ (see Section 4 for detail). Lind and Ward showed that the (measure theoretic) entropy of the $\mathbb{Z}$-action on $\Sigma^m$ as above is given by
\begin{align}\label{explicitLW}
h(\Sigma^m)=\sum_{l \leq \infty} \sum_{|\lambda|_l>1} \log |\lambda|_l \in \mathbb{R}
\end{align}
where $l \leq \infty$ means that $l$ runs over all places of $\mathbb{Q}$ and $\lambda$ runs over all eigenvalues of $A$ with $|\lambda|_l>1$ \cite[Theorems 1, 2]{LW88}. Our goal is to obtain a $p$-adic analogue of this theorem. However, it is not straightforward because the entropy in Lind-Ward's theorem is measure theoretic and we do not have $p$-adic analogue of measure theoretic entropy. Furthermore, if $S$ is infinite, above $\mathbb{Z}$-action on $\Sigma_S^m$ is neither expansive nor $p$-adically expansive. Now, we will modify the dynamical system to make the action $p$-adically expansive and prove the following theorem.

\begin{theorem}\label{thmK2}
Let $S$ be a finite set of primes and we fix $A \in \mathrm{GL}_m \left(\mathbb{Z}\left[1/{S}\right]\right)$ and embeddings $\overline{\mathbb{Q}} \hookrightarrow \mathbb{C}_p$ for all $p \in S$. We see $\mathbb{Z}\left[1/{S}\right]^m$ as a $\mathbb{Z}[t^{\pm 1}]$-module with the structure induced by above $\mathbb{Z}$-action on $\mathbb{Z}\left[1/{S}\right]^m$. Assume that the eigenvalues $ \lambda_1,\dots,\lambda_m$ of $A$ satisfy $|\lambda_k|_p \neq 1$ for all $k=1,\dots,m$ and all $p \in S$. Then the following properties hold $\colon$
\begin{enumerate}
\item As a $\mathbb{Z}[t^{\pm 1}]$-module, $\mathbb{Z}\left[1/{S}\right]^m$ is finitely generated and the $\mathbb{Z}$-action on $\Sigma_S^m$ is $p$-adically expansive.
\item The $p$-adic entropy $h_p(\Sigma_S^m)$ of the $\mathbb{Z}$-action exists for all $p \in S$ and we have
\begin{align*}
h_p(\Sigma_S^m)= \displaystyle \sum_{\substack{l \in S \\ l \neq p}} \displaystyle \sum_{|\lambda_k|_l >1} \log_p |\lambda_k|_l + \displaystyle \sum_{|\lambda_k|_p >1} \log_p \lambda_k.
\end{align*}
\end{enumerate}
Here, we see $\lambda_k \in \mathbb{C}_p$ under the fixed embedding $\overline{\mathbb{Q}} \hookrightarrow \mathbb{C}_p$ for all $p \in S$.
\end{theorem}

\vspace{10pt}
\noindent
\textsc{Notation:}~ In this paper, let $p$ be a prime, $\mathbb{C}_p$ be the completion of $\overline{\mathbb{Q}}_p$ with the norm $|p|_p=p^{-1}$ and $\log_p$ be the $p$-adic logarithm with $\log_p p=0$. For a positive integer $d$ and a commutative ring $A$, we write $L_d(A)=A[t_1^{\pm 1},\cdots,t_d^{\pm 1}]$. For a locally compact abelian group $M$, we denote its Pontryagin dual by $\hat{M}$ or $M^\wedge$.

\vspace{10pt}
\noindent
\textsc{Acknowledgment:}~ The author would like to thank my supervisor Professor Takao Yamazaki so much for his advice and helpful comments. This paper is based on the author's master thesis. This work was supported in part by the WISE Program for AI Electronics, Tohoku University.

\section{Proof of Theorem \ref{thmK1}}

In this section, we will prove Theorem \ref{thmK1}.

\begin{lemma}\label{pre1}
Let $L/K$ be a finite extension of fields and $V$ be a finite dimensional $L$-vector space. For any $f \in {\rm End}_L(V)$, we have
\begin{align*}
{\rm det}_K f=N_{L/K}({\rm det}_L f).
\end{align*}
\end{lemma}

The proof of Lemma \ref{pre1} is given by an elementary linear algebra and hence omitted.

\begin{corollary}\label{pre2}
Let $K$ be a number field and $F_1$ (resp. $F_2$) be the fractional field of $L_d(\mathbb{Q})$ (resp. $L_d(K)$). For any $f \in {\rm GL}_s(F_2)$, we have
\begin{align*}
N_{K/\mathbb{Q}}({\rm det}_{F_2}f)={\rm det}_{F_1}f.
\end{align*}
Here, for ${\rm Hom}_{\mathbb{Q}}(K,\overline{\mathbb{Q}})=\{\tau_1, \cdots ,\tau_r\}$ and $P=\sum_{\nu \in \mathbb{Z}^d} a_{\nu}t^{\nu} \in F_2$, we define $\tau_i(P)=\sum_{\nu \in \mathbb{Z}^d} \tau_i(a_{\nu})t^{\nu} \in F_2,N_{K/\mathbb{Q}}(P)=\prod_{i=1}^r \tau_i(P)$.
\end{corollary}

\begin{proof}
We may identify ${\rm Hom}_{F_1}(F_2,\overline{F_1})$ and $N_{F_2/F_1}$ with ${\rm Hom}_{\mathbb{Q}}(K,\overline{\mathbb{Q}})$ and $N_{K/\mathbb{Q}}$ respectively. Using Lemma \ref{pre1}, we get
\begin{align*}
{\rm det}_{F_1}f={\rm det}_{F_1}({\rm det}_{F_2}f)=N_{K/\mathbb{Q}}({\rm det}_{F_2}f).
\end{align*}
\end{proof}

\begin{theorem}\cite[Theorem 3.2]{D09}\label{thmD'}
Assume that $f \in M_r(L_d(\mathbb{Z}))$ and $\det f$ does not vanish at any point of the $p$-adic $d$-torus $T_p^d$ given by (\ref{p-adic torus}). Then the $p$-adic entropy $h_p(X_f)$ of the $\mathbb{Z}^d$-action on $X_f \coloneqq (L_d(\mathbb{Z})^r/{fL_d(\mathbb{Z})^r})^\wedge$ exists and we have
\begin{align*}
h_p(X_f)=m_p(\det f).
\end{align*}
Here, the $\mathbb{Z}^d$-action on $X_f$ is induced by the $\mathbb{Z}^d$-action on $L_d(\mathbb{Z})^r$ given by
\begin{align}\label{actiononX_f'}
\delta \cdot \left( \displaystyle \sum_{\nu \in \mathbb{Z}^d} a_\nu^{(1)} t^\nu, \cdots, \sum_{\nu \in \mathbb{Z}^d} a_\nu^{(r)} t^\nu \right)= \left( \displaystyle \sum_{\nu \in \mathbb{Z}^d} a_\nu^{(1)} t^{\nu + \delta}, \cdots ,\sum_{\nu \in \mathbb{Z}^d} a_\nu^{(r)} t^{\nu + \delta} \right)
\end{align}
and $h_p(X_f)$ is given by Definition \ref{def of p-adic entropy}.
\end{theorem}

Theorem \ref{thmK1} follows immediately from the following theorem.

\begin{theorem}\label{thmK1'}
Let $K$ be a number field and $\mathcal{O}_K$ be the ring of integers of $K$. Assume that $f \in M_s(L_d(\mathcal{O}_K))$ and $N_{K/\mathbb{Q}}(\det f)$ does not vanish at any point of the $p$-adic $d$-torus $T_p^d$. Then the $p$-adic entropy $h_p(X_f)$ of the $\mathbb{Z}^d$-action on $X_f\coloneqq (L_d(\mathcal{O}_K)^s/{fL_d(\mathcal{O}_K)^s})^\wedge$ exists and we have
\begin{align*}
h_p(X_f)=m_p(N_{K/\mathbb{Q}}(\det f)).
\end{align*}
Here, the $\mathbb{Z}^d$-action on $X_f$ is induced by the $\mathbb{Z}^d$-action on $L_d(\mathbb{Z})^r$ as (\ref{actiononX_f'}).
\end{theorem}

\begin{proof}
Let $[K \colon \mathbb{Q}]=r$. For $f \in M_s(L_d(\mathcal{O}_K))$, we define $L_d(\mathbb{Z})$-linear homomorphism
\begin{align*}
\varphi_f \colon L_d(\mathcal{O}_K)^s \rightarrow L_d(\mathcal{O}_K)^s
\end{align*}
by
\begin{align*}
x \mapsto fx
\end{align*}
and let $A_f \in M_{rs}(L_d(\mathbb{Z}))$ be its matrix representation. Since
\begin{align*}
L_d(\mathcal{O}_K)^s/{f \cdot L_d(\mathcal{O}_K)^s}={\rm Coker}(\varphi_f) \simeq {\rm Coker}(A_f)=L_d(\mathbb{Z})^{rs}/{A_fL_d(\mathbb{Z})^{rs}},
\end{align*}
we get
\begin{align*}
\left|{\rm Fix}_{(n \mathbb{Z})^d} \left(L_d(\mathcal{O}_K)^s/{f \cdot L_d(\mathcal{O}_K)^s}\right)^\wedge \right|=\left|{\rm Fix}_{(n \mathbb{Z})^d} \left(L_d(\mathbb{Z})^{rs}/{A_fL_d(\mathbb{Z})^{rs}}\right)^\wedge \right|.
\end{align*}
Theorem \ref{thmD'} and Corollary \ref{pre2} imply
\begin{align*}
h_p(X_f)=h_p(X_{A_f})=m_p(\det A_f)=m_p(N_{K/\mathbb{Q}}(\det f)).
\end{align*}
\end{proof}

\begin{example}
Let $K=\mathbb{Q}(\sqrt{2})$ and $f=3t+\sqrt{2} \in \mathcal{O}_K [t^{\pm 1}]$. Since the equation
\begin{align*}
N_{K/\mathbb{Q}}(f)=9t^2-2=0
\end{align*}
has the roots $t=\pm \sqrt{2}/{3}$, we have
\begin{equation*}
\left|\frac{\sqrt{2}}{3}\right|_p=
\begin{cases}
  1  &  (p \neq 2,3) \\
  2^{-\frac{1}{2}}  &  (p=2) \\
  3  &  (p=3).
\end{cases}
\end{equation*}
If $p=2,3$, there exists the $p$-adic entropy of $f$ and we have
\begin{eqnarray*}
h_2(X_f)   &=&   m_2(9t^2-2)   =   \log_2 9 \in \mathbb{C}_2 , \\
h_3(X_f)   &=&   m_3(9t^2-2)   =   \log_3 2 \in \mathbb{C}_3.
\end{eqnarray*}
Note that
\begin{align*}
h(X_f)=m(9t^2-2)=\log 9 \in \mathbb{R}.
\end{align*}
\end{example}

\section{$p$-adically expansiveness}

In this section, we will recall the (classical) expansiveness and explain the $p$-adically expansiveness, which is introduced by Br\"auer. Let $\Gamma$ be a countable discrete group in this section.

\begin{definition}
Let $X$ be a compact metrizable topological space and assume that $\Gamma$ acts on $X$. A continuous action of $\Gamma$ on $X$ is expansive if  the following holds: There is a metric $d$ defining the topology of $X$ and some $\epsilon >0$ such that for every pair of distinct points $x \neq y$ in $X$ there exists an element $\gamma \in \Gamma$ with $d(\gamma x,\gamma y) \geq \epsilon$.
\end{definition}

It is known that if the action of $\Gamma$ on $X$ is expansive, for every cofinite normal subgroup $\Lambda$ of $\Gamma$, ${\rm Fix}_{\Lambda}(X)$ is finite. Moreover, for the case $\Gamma=\mathbb{Z}^d$ the following holds.

\begin{theorem}\cite[Theorem A.1]{LSW90}
Let $X$ be a compact metrizable abelian group and assume that $\mathbb{Z}^d$ acts expansively on $X$. Then, topological entropy, measure theoretic entropy and periodic entropy coincide.
\end{theorem}

To describe the $p$-adically expansiveness, we consider equivalent conditions to the expansiveness.

\begin{definition}
If $M$ is a discrete left $\mathbb{Z}[\Gamma]$-module, the $\Gamma$-action on $M$ induces the action on the Pontryagin dual $X \coloneqq \hat{M}$. Conversely, if $X$ is a compact $\Gamma$-module, the $\mathbb{Z}[\Gamma]$-module structure of $M \coloneqq \hat{X}$ is induced. In particular, for $f \in \mathbb{Z}[\Gamma]$ and $M=\mathbb{Z}[\Gamma]/{\mathbb{Z}[\Gamma]f}$, we write $X=X_f \coloneqq \hat{M}$.
\end{definition}

\begin{theorem}\cite[Theorem3.2]{DS07}\label{expansive1}
Let $f \in \mathbb{Z}[\Gamma]$. Then, the $\Gamma$-action on $X_f$ is expansive if and only if $f \in L^1(\Gamma)^{\times}$. Here, $L^1(\Gamma)$ is an algebra given by
\begin{align*}
L^1(\Gamma)=\left\{w=(w_{\gamma})_{\gamma \in \Gamma} \in \prod_{\Gamma} \mathbb{R} \middle | \|w\|_1=\sum_{\gamma \in \Gamma}|w_{\gamma}| < \infty \right\}.
\end{align*}
\end{theorem}

\begin{theorem}\cite[Theorem3.1]{CL15}\label{expansive2}
Let $X$ be a compact $\Gamma$-module. Then, the $\Gamma$-action on $X$ is expansive if and only if $M=\hat{X}$ is a finitely generated $\mathbb{Z}[\Gamma]$-module satisfying $L^1(\Gamma) \otimes_{\mathbb{Z}[\Gamma]} M=0$.
\end{theorem}

In the case $\Gamma=\mathbb{Z}^d$, the following holds.

\begin{theorem}\cite[Corollary3.2]{CL15}\cite[Theorem6.5]{Sc95}\label{expansive3}
Let $X$ be a compact metrizable $\mathbb{Z}^d$-module and assume that $M=\hat{X}$ is a finitely generated $L_d(\mathbb{Z})=\mathbb{Z}[t_1^{\pm 1},\cdots,t_d^{\pm 1}] \simeq \mathbb{Z}[\mathbb{Z}^d]$-module. Then, the following conditions are equivalent:
\begin{enumerate}
\item The $\mathbb{Z}^d$-action on $X$ is expansive.
\item For every $\mathfrak{p} \in {\rm Ass}(M)$, we have $V_{\mathbb{C}}(\mathfrak{p}) \cap T^d = \emptyset$.
\item The module $M$ is $S_{\infty}$-torsion, where $S_{\infty} \subset L_d(\mathbb{Z})$ is the multiplicative system $S_{\infty}=L_d(\mathbb{Z}) \cap L^1(\mathbb{Z}^d)^{\times}$.
\end{enumerate}
Here, $T^d$ is the $d$-torus given by (\ref{torus}) and define
\begin{align*}
{\rm Ass}(M) &=\left\{\mathfrak{p} \in {\rm Spec}L_d(\mathbb{Z}) \middle | \mathfrak{p}={\rm ann}(a) \ \text{for some} \ a \in M  \right\}, 
\end{align*}
and for $\mathfrak{p} \in {\rm Spec}L_d(\mathbb{Z})$ and for $K=\mathbb{C}$ or $\overline{\mathbb{Q}}_p$,
\begin{align*}
V_{K}(\mathfrak{p})&=\left\{ z \in \left(K^{\times}\right)^d \middle | f(z)=0 \ \text{for every} \ f \in \mathfrak{p} \right\}.
\end{align*}
\end{theorem}

We will summarize the properties of the $p$-adically expansiveness by considering these theorems.

\begin{definition}
The algebra $c_0(\Gamma)$ is defined by
\begin{align*}
c_0(\Gamma) \coloneqq \left\{\sum_{\gamma \in \Gamma}x_{\gamma}\gamma \in \mathbb{Q}_p[[\Gamma]] \middle| |x_{\gamma}|_p \to 0 \ \text{as} \ \gamma \to \infty \ \ (*) \right\},
\end{align*}
where (*) means that for any $\epsilon >0$ there exists a finite set $S\subset \Gamma$ such that $|x_{\gamma}|_p<\epsilon$ holds for every $\gamma \in \Gamma \setminus S$. For two elements $\sum_{\gamma \in \Gamma}x_{\gamma}\gamma$ and $\sum_{\gamma \in \Gamma}y_{\gamma}\gamma \in c_0(\Gamma)$, the product is given by
\begin{align*}
\sum_{\gamma \in \Gamma}x_{\gamma}\gamma \cdot \sum_{\gamma \in \Gamma}y_{\gamma}\gamma \coloneqq \sum_{\gamma \in \Gamma} \left(\sum_{\delta \in \Gamma}x_{\delta}y_{\delta^{-1}\gamma}\right)\gamma \in c_0(\Gamma).
\end{align*}
The algebra $c_0(\Gamma)$ is equipped with a norm
\begin{align*}
\| \cdot \| : c_0(\Gamma) \rightarrow \mathbb{R}_{\geq 0} \ ; \ \sum_{\gamma \in \Gamma}x_{\gamma}\gamma \mapsto \max_{\gamma \in \Gamma}\{|x_{\gamma}|_p\}.
\end{align*}
\end{definition}

One can check that $(c_0(\Gamma),\| \cdot \|)$ is a $p$-adic Banach algebra over $\mathbb{Q}_p$. See \cite[Section 2]{D09} for the definition of the $p$-adic Banach algebra and details of $c_0(\Gamma)$.
Note that we see $c_0(\Gamma)$ as a $p$-adic analogue of $L^1(\Gamma)$ and we have the correspondence
\begin{align*}
\begin{array}{ccc}
L^1(\Gamma) & \longleftrightarrow & c_0(\Gamma) \\
\rotatebox{90}{$\subset$} & & \rotatebox{90}{$\subset$} \\
\mathbb{C}[\Gamma] & \longleftrightarrow & \mathbb{Q}_p[\Gamma]
\end{array}
\end{align*}
between the archimedian algebra and non-archimedian one. We replace $L^1(\Gamma)$ with $c_0(\Gamma)$ in Theorem \ref{expansive1}, Theorem \ref{expansive2} and Theorem \ref{expansive3} as follow.

\begin{definition}
Let $X$ be an abelian group and $p$ be a prime. Then $X$ is said to have bounded $p$-torsion if there exists an integer $i_0 \geq 0$ such that ${\rm Ker}(p^i : X \rightarrow X)={\rm Ker}(p^{i_0} : X \rightarrow X)$ for all $i \geq i_0$.
\end{definition}

\begin{theorem}\label{thmBr1}\cite[Theorem14]{D12}
Let $f \in \mathbb{Z}[\Gamma]$. The following conditions are equivalent:
\begin{enumerate}
\item The abelian group $X_f=(\mathbb{Z}[\Gamma]/{\mathbb{Z}[\Gamma]f})^\wedge$ has bounded $p$-torsion.
\item There exists an element $g \in c_0(\Gamma)$ such that $gf=1$.
\item For $M_f=\mathbb{Z}[\Gamma]/{\mathbb{Z}[\Gamma]f}$, we have $c_0(\Gamma) \otimes_{\mathbb{Z}[\Gamma]} M_f=0$.
\end{enumerate}
In this case, ${\rm Fix}_N(X_f)$ is finite for any cofinite normal subgroup $N \subset \Gamma$.
\end{theorem}

\begin{remark}
The condition 2 in Theorem \ref{thmBr1} implies that $f \in c_0(\Gamma)^{\times}$. In other words, if $fg=1$ for $f,g \in c_0(\Gamma)$, then we have $gf=1$. This answers a question raised by Deninger \cite[Section 3]{D12}. Indeed, for $f=\sum_{\gamma \in \Gamma}x_{\gamma}\gamma,g=\sum_{\gamma \in \Gamma}y_{\gamma}\gamma$, since
\begin{align*}
fg=\sum_{\gamma \in \Gamma}x_{\gamma}\gamma \cdot \sum_{\gamma \in \Gamma}y_{\gamma}\gamma= \sum_{\gamma \in \Gamma} \left(\sum_{\delta \in \Gamma}x_{\delta}y_{\delta^{-1}\gamma}\right)\gamma=1,
\end{align*}
it follows that
\begin{align}\label{conj1}
\begin{cases}
\sum_{\delta \in \Gamma}x_{\delta}y_{\delta^{-1}}=1 & (\gamma=1) \\
\sum_{\delta \in \Gamma}x_{\delta}y_{\delta^{-1}\gamma}=0 & (\gamma \neq 1).
\end{cases}
\end{align}
Now, since
\begin{align*}
gf=\sum_{\gamma \in \Gamma} \left(\sum_{\delta \in \Gamma}y_{\delta}x_{\delta^{-1}\gamma}\right)\gamma =\sum_{\gamma \in \Gamma} \left(\sum_{\delta \in \Gamma}x_{\delta}y_{\gamma\delta^{-1}}\right)\gamma,
\end{align*}
it is enough to show that
\begin{align}\label{conj2}
\begin{cases}
\sum_{\delta \in \Gamma}x_{\delta}y_{\delta^{-1}}=1 & (\gamma=1) \\
\sum_{\delta \in \Gamma}x_{\delta}y_{\gamma\delta^{-1}}=0 & (\gamma \neq 1).
\end{cases}
\end{align}
The first equation in (\ref{conj2}) follows from the first one in (\ref{conj1}). We obtain the second equation in (\ref{conj2}) by substituting $\delta\gamma\delta^{-1} (\neq 1)$ for $\gamma$ in the second one in (\ref{conj1}).
\end{remark}

\begin{definition}\label{p-expansive1}\cite{B10}
The $\Gamma$-action on $X_f$ is $p$-adically expansive if either condition 1-3 in Theorem \ref{thmBr1} is satisfied.
\end{definition}

The following theorem is a $p$-adic analogue of Theorem \ref{expansive3}.

\begin{theorem}\label{thmBr2}\cite[Proposition 4.19, Proposition 4.22]{B10}
Let $X$ be a compact $\mathbb{Z}^d$-module and assume that $M=\hat{X}$ is a finitely generated $L_d(\mathbb{Z})$-module. Then, the following conditions are equivalent:
\begin{enumerate}
\item The abelian group $X$ has bounded $p$-torsion.
\item For every $\mathfrak{p} \in {\rm Ass}(M)$, we have $V_{\overline{\mathbb{Q}}_p}(\mathfrak{p}) \cap T_p^d = \emptyset$.
\item The module $M$ is $S_p$-torsion, where $S_p \subset L_d(\mathbb{Z})$ is the multiplicative system $S_p=L_d(\mathbb{Z}) \cap c_0(\mathbb{Z}^d)^{\times}$.
\end{enumerate}
Here, $T_p^d$ is the $p$-adic torus given by (\ref{p-adic torus}).
\end{theorem}

\begin{definition}\cite{B10}
The $\mathbb{Z}^d$-action on $X$ is $p$-adically expansive if either condition 1-3 in Theorem \ref{thmBr2} is satisfied. This is compatible with Definition \ref{p-expansive1}.
\end{definition}

\begin{remark}
Does there exist $p$-adic entropy if the action is $p$-adically expansive? This is not true in general. An example can be found in \cite[Example 7.1]{B10}. However there exists an example of a $p$-adically expansive $\mathbb{Z}^d$-action whose $p$-adic entropy exists \cite[Theorem 1.1]{D09}. It is an open problem to define a better notion of $p$-adic entropy or $p$-adically expansiveness. See \cite{B10} and \cite{D12} for detail.
\end{remark}

\section{$p$-adic entropy of solenoids}

In this section, we will prove Theorem \ref{thmK2}. First, we will explain the dynamical system which we consider. Let $m \in \mathbb{N}$. For a subset $S$ of primes, let $M=\mathbb{Z}\left[1/{S}\right]^m, X=\Sigma_S^m$ ($\Sigma_S=\mathbb{Z}\left[1/{S}\right]^\wedge$ is a solenoid) and fix $A \in {\rm GL}_m\left(\mathbb{Z}\left[1/{S}\right]\right)$. The $\mathbb{Z}$-action on $M$ given by
\begin{align*}
n \cdot x \coloneqq A^nx
\end{align*}
induces the $\mathbb{Z}$-action on $X$. As mentioned in Section 3, $M$ also has the $\mathbb{Z}[\mathbb{Z}]=\mathbb{Z}[t^{\pm 1}]$-module structure defined by
\begin{align*}
f \cdot x \coloneqq f(A)x.
\end{align*}
We will consider a $p$-adic analogue of Lind-Ward's theorem (see (\ref{explicitLW})) by modifying the dynamical system and prove Theorem \ref{thmK2}.

\begin{lemma}\label{lemma4.1}
Let $S$ be a finite set of primes and $A \in M_m \left(\mathbb{Z}\left[1/{S}\right]\right)$. Assume that $\det A \neq 0$. Then we have
\begin{align*}
\left|\mathbb{Z}\left[\frac{1}{S}\right]^m/{A\mathbb{Z}\left[\frac{1}{S}\right]^m}\right|=\left( \prod_{l \in S} |\det A|_l \right) \cdot |\det A|.
\end{align*}
\end{lemma}

\begin{proof}
Since $\mathbb{Z}\left[1/{S}\right]$ is a principal ideal domain, there exist matrices $P, Q \in \mathrm{GL}_m \left(\mathbb{Z}\left[1/{S}\right]\right)$ such that
\begin{align*}
QAP=\mathrm{diag}(e_1,e_2, \cdots ,e_m).
\end{align*}
We may assume that $e_1,\dots,e_m \in \mathbb{N}$ and $e_i$ is prime to $p$ for all $1\leq i \leq m$ and all $p \in S$. Because
\begin{align*}
\mathbb{Z}\left[\frac{1}{S}\right]^m/{A\mathbb{Z}\left[\frac{1}{S}\right]^m} \simeq \displaystyle \prod_{i=1}^m \mathbb{Z}\left[\frac{1}{S}\right]/{e_i\mathbb{Z}\left[\frac{1}{S}\right]} \simeq \displaystyle \prod_{i=1}^m \mathbb{Z}/{e_i\mathbb{Z}},
\end{align*}
we have
\begin{align*}
\left|\mathbb{Z}\left[\frac{1}{S}\right]^m/{A\mathbb{Z}\left[\frac{1}{S}\right]^m}\right|=\displaystyle \prod_{i=1}^m e_i=\displaystyle \prod_{l \in S} l^{a_l} \cdot |\det A|
\end{align*}
for some $a_l \in \mathbb{Z}$. On the other hand, since $e_i$ is prime to $p$ for all $1\leq i \leq m$ and all $p \in S$, we see that
\begin{align*}
v_l(l^{a_l} \cdot \det A)&=v_l\left(\prod_{l \in S} l^{a_l} \cdot |\det A|\right) \\
                                 &=v_l\left(\prod_{i=1}^m e_i \right)=0,
\end{align*}
where $v_l$ is the $l$-adic valuation. This implies $a_l=-v_l(\det A)$.
\end{proof}

\begin{proof}[Proof of Theorem \ref{thmK2}]
1. We write $M=\mathbb{Z}\left[1/{S}\right]^m$ and $R=\mathbb{Z}[t^{\pm 1}]$. Let $\bm{e}_k$ be the standard basis for $M$. We will show that
\begin{align}\label{proof thm1.5}
M=R\bm{e}_1+ \cdots +R\bm{e}_m,
\end{align}
but it is enough to show that for all $1 \leq k \leq m$ and $\{j_l\}_{l \in S} \subset \prod_{l \in S} \mathbb{Z}_{\geq 0}$, $\left(\prod_{l \in S} l^{-j_l}\right)\bm{e}_k$ is contained in the R.H.S. of (\ref{proof thm1.5}). Put $\chi_A(u)=\det (uI-A)=\sum_{i=0}^m \alpha_iu^i \in \mathbb{Z}\left[1/{S}\right][u]$. Then the assumption is equivalent to the condition that for all $p \in S$ the Newton polygon of $\chi_A$ does not have the segments with slope 0. This means that there exists exactly one index which maximizes $\{|\alpha_m|_p, |\alpha_{m-1}|_p, \dots,|\alpha_0|_p\}$. Denote the index which maximizes $\{|\alpha_m|_p, |\alpha_{m-1}|_p, \dots,|\alpha_0|_p\}$ by $i_0$ and we write $\alpha_{i_0}=p^{-e_{i_0}}\beta_{i_0}$, where $\beta_{i_0} \in \mathbb{Z}\left[1/{S}\right], |\beta_{i_0}|_p=1$ and $e_{i_0} \in \mathbb{Z}_{\geq 0}$. Using the Cayley-Hamilton theorem, we get
\begin{align*}
\displaystyle \sum_{\substack{0 \leq i \leq m \\ i \neq i_0}} \alpha_iA^i+p^{-e_{i_0}}\beta_{i_0}A^{i_0}=O.
\end{align*}
Here, we identify $A^0$ with the identity matrix $I$. Multiplying this by some rationals and $A^{-i_0}$, we obtain
\begin{align*}
p^{-1}\gamma_{i_0}I=-\displaystyle \sum_{i \neq i_0} c_iA^{i-i_0},
\end{align*}
where $\gamma_{i_0}=\beta_{i_0}\displaystyle \prod_{\substack{l \in S \\ l \neq p}} l^{-f_l}, c_i=p^{e_{i_0}-1}\alpha_i\displaystyle \prod_{\substack{l \in S \\ l \neq p}} l^{-f_l}$ for all $i \neq i_0$ and $f_l=\displaystyle\min_{0 \leq i \leq m} \{v_l(\alpha_i), \ 0\}$ for all $l \in S$. We can check that $\gamma_{i_0}, c_i \in \mathbb{Z}$ and $\gamma_{i_0}$ is prime to $p$. Since $\gamma_{i_0}$ is prime to $p$, we can take $x_{i_0}, y_{i_0} \in \mathbb{Z}$ satisfying $\gamma_{i_0}x_{i_0}+py_{i_0}=1$. Thus we get
\begin{align*}
p^{-1}I&=p^{-1}(\gamma_{i_0}x_{i_0}+py_{i_0})I\\
         &=-x_{i_0}\displaystyle\sum_{i \neq i_0}c_iA^{i-i_0} +y_{i_0}I.
\end{align*}
We take $g_p(t)=-x_{i_0}\displaystyle\sum_{i \neq i_0}c_it^{i-i_0} +y_{i_0} \in \mathbb{Z}[t^{\pm 1}]$ and then it follows that
\begin{align*}
\left(\prod_{p \in S} g_p(A)^{j_p}\right)\bm{e}_k=\left(\prod_{l \in S} l^{-j_l}\right)\bm{e}_k
\end{align*}
2. Using Lemma \ref{lemma4.1}, we compute
\begin{align*}
|\mathrm{Fix}_{n\mathbb{Z}}(X)|&=|\mathrm{Ker} {(I-A^n)}| \\
                               &=|(\mathrm{Coker}(I-A^n))^\wedge| \\
                               &=\left|\mathbb{Z}\left[\frac{1}{S}\right]^m/{(I-A^n)\mathbb{Z}\left[\frac{1}{S}\right]^m}\right| \\
                               &=\displaystyle \prod_{l \in S} l^{-v_l(\det (I-A^n))} \cdot \left|\displaystyle\prod_{k=1}^m (1-\lambda_k^n)\right|
\end{align*}
and
\begin{align}
&\frac{1}{n} \log_p \left(\displaystyle \prod_{l \in S} l^{-v_l(\det (I-A^n))} \cdot \left|\displaystyle\prod_{k=1}^m (1-\lambda_k^n)\right| \right) \\
=&-\displaystyle\sum_{\substack{l \in S \\ l \neq p}} \frac{1}{n}v_l(\det (I-A^n)) \log_pl + \displaystyle\sum_{k=1}^m \frac{1}{n} \log_p(1-\lambda_k^n).\label{comK2}
\end{align}
We can check that the second term of (\ref{comK2}) is convergent to $\displaystyle \sum_{|\lambda_k|_p >1} \log_p \lambda_k$ as $n \to \infty$. In the first term of (\ref{comK2}), since 
\begin{align*}
\begin{cases}
v_l(1-\lambda_k^n)=0 & \text{if} \ |\lambda_k|_l<1 \\
v_l(1-\lambda_k^n)=v_l(\lambda_k^n)=n \cdot v_l(\lambda_k) & \text{if} \ |\lambda_k|_l>1,
\end{cases}
\end{align*}
it follows that
\begin{align*}
\text{the first term of (\ref{comK2})} =&-\displaystyle\sum_{\substack{l \in S \\ l \neq p}}\sum_{k=1}^m v_l(1-\lambda_k^n) \frac{1}{n} \log_pl  \\
=&-\displaystyle\sum_{\substack{l \in S \\ l \neq p}}\sum_{|\lambda_k|_l>1} v_l(\lambda_k)  \log_pl \\
=&\displaystyle \sum_{\substack{l \in S \\ l \neq p}} \displaystyle \sum_{|\lambda_k|_l >1} \log_p |\lambda_k|_l.
\end{align*}
\end{proof}

\begin{example}
Let $S=\{2,3\}$ and
$$
A=\left(
\begin{array}{ccc}
\frac{1}{6} & 1 \\
-1            & 0
\end{array}
\right) \in {\rm GL}_2 \left(\mathbb{Z}\left[\frac{1}{6}\right]\right).
$$
The eigenvalues of $A$ are $\lambda_{\pm} = \frac{1 \pm \sqrt{-143}}{12}$. Since $|\lambda_{\pm}|_p=p^{\pm 1}$ for $p=2,3$, $p$-adic entropies are as follow:
\begin{align*}
h_2(A;X)=&\log_2 3+\log_2 \lambda_+ \in \mathbb{C}_2, \\
h_3(A;X)=&\log_3 2+\log_3 \lambda_+ \in \mathbb{C}_3. 
\end{align*}
Note that
\begin{align*}
h(A;X)=\log 2 +\log 3 =\log 6 \in \mathbb{R}.
\end{align*}
\end{example}

\vspace{10pt}
\noindent
Mathematical Institute, Graduate School of Science, Tohoku University,\\
6-3 Aramakiaza, Aoba, Sendai, Miyagi 980-8578, Japan.\\
E-mail address: \textbf{yu.katagiri.s3@dc.tohoku.ac.jp}

\end{document}